\documentclass[11pt]{amsart}
\usepackage{amssymb,pstricks,pst-plot}
\parskip=\smallskipamount

\newcommand{\C}{\mathbb{C}}
\newcommand{\D}{\mathbb{D}}
\newcommand{\clD}{{\overline{\,\D}}}

\newcommand{\R}{\mathbb{R}}

\newcommand{\bT}{\mathbb{T}}

\renewcommand{\d}{\mathrm{d}}

\newcommand\E{\mathrm{e}}
\newcommand\I{\mathrm{i}}

\newcommand{\cA}{\mathcal{A}}

\newcommand{\cC}{\mathcal{C}}

\newcommand{\cO}{\mathcal{O}}

\newcommand\wt{\widetilde}

\newcommand{\Psh}{\mathrm{Psh}}

\newcommand{\ol}{\overline}

\newcommand\wh{\widehat}
\newcommand\reg{\mathrm{reg}}
\newcommand\sing{\mathrm{sing}}
\newcommand\bs{\backslash}

\newcommand\di{\partial}

\newcommand\dist{\mathrm{dist}}
\newcommand\dibar{\bar\partial}

\newtheorem{theorem}{Theorem}[section]
\newtheorem{lemma}[theorem]{Lemma}

\newtheorem{corollary}[theorem]{Corollary}
\newtheorem{proposition}[theorem]{Proposition}

\theoremstyle{definition}
\newtheorem{definition}[theorem]{Definition}
\newtheorem{example}[theorem]{Example}

\newtheorem{remark}[theorem]{Remark}

\def\ss{\Subset}

\begin{document}
\title{The Poletsky-Rosay theorem on singular complex spaces}
\author{Barbara Drinovec Drnov\v sek \& Franc Forstneri\v c}
\address{Faculty of Mathematics and Physics, University of Ljubljana, 
and Institute of Mathematics, Physics and Mechanics, Jadranska 19, 
1000 Ljubljana, Slovenia}
\email{barbara.drinovec@fmf.uni-lj.si}
\email{franc.forstneric@fmf.uni-lj.si}
\thanks{Research supported by grants P1-0291 and J1-2152, Republic of Slovenia.}

%
%

\date{\today} 

%
%
%
%
\begin{abstract}
In this paper we extend the Poletsky-Rosay theorem, 
concerning plurisubharmonicity of the Poisson envelope
of an upper semicontinuous function, 
to locally irreducible complex spaces.
\end{abstract}

\maketitle

\section{Introduction}
Plurisubharmonic functions were introduced by 
K.\ Oka \cite{Oka1942} and P.\ Lelong \cite{Lelong1942} in 1942;
ever since then they have been playing a major role in complex analysis.  
The minimum of two plurisubharmonic functions is not plurisubharmonic 
in general. There has been a considerable amount  
of interest in studying situations where the 
infimum actually is plurisubharmonic. The first major result of this kind 
was Kiselman's  {\em minimum principle} \cite{Kiselman1978}. 

In the early 1990's E.\ Poletsky \cite{Poletsky1991,Poletsky1993} 
found a novel way of constructing plurisubharmonic functions 
as pointwise infima of upper semicontinuous functions. 
Set $\D=\{\zeta\in \C\colon |\zeta|< 1\}$ and 
$\bT=b\D=\{\zeta\in\C\colon |\zeta|=1\}$. Let $\cA(\D,X)$ denote the 
set of all {\em analytic discs} in a complex space $X$, 
that is, continuous maps $\clD \to X$ that are holomorphic in $\D$.
Then for $x\in X$ let $\cA(\D,X,x)=\{f\in \cA(\D,X)\colon f(0)=x\}$. 
Our main result is the following.

%
\begin{theorem}
\label{Poletsky-Rosay}
Let $(X,\cO_X)$ be an irreducible and locally irreducible 
(reduced, paracompact) complex space,
and let $u\colon X\to \R\cup\{-\infty\}$ be an upper semicontinuous 
function on $X$. Then the function 
\begin{equation}
\label{eq:Poisson-funct}
   \wh u(x) = \inf \Big\{\int^{2\pi}_0 u(f(\E^{\I t}))\, 
   \frac{\d t}{2\pi} \colon \ f\in \cA(\D,X,x) \Big\},
    \quad x\in X
\end{equation}
is plurisubharmonic on $X$ or identically $-\infty$; 
moreover, $\wh u$ is the supremum of the plurisubharmonic 
functions on $X$ which are not greater than $u$. 
\end{theorem}

In the basic case when $X=\C^n$ this was proved by 
Poletsky \cite{Poletsky1991,Poletsky1993} 
and by Bu and Schachermayer \cite{Bu-Schachermayer}.
The result was extended to some complex manifolds 
(and to certain other disc functionals)
by L\'arusson and Sigurdsson 
\cite{Larusson-Sigurdsson1998,Larusson-Sigurdsson2003,Larusson-Sigurdsson2009}, 
and to all complex manifolds by Rosay \cite{Rosay:Poletsky1,Rosay:Poletsky2}.
In this paper we use the method of gluing sprays,  
developed in \cite{BDF1,FFAsian}, to give a new proof
which applies to any normal complex space $X$. 
For a locally irreducible space 
the result follows easily by using seminormalization to 
reduce to the case of a normal space. 

A reduction to the Poletsky-Rosay theorem on manifolds 
is also possible by appealing to the Hironaka desingularization 
theorem, but this approach seems unreasonable since 
our proof is no more difficult than the 
original proofs given for manifolds. 
%
%
%
%
%
\begin{example}
\label{counterex}
The theorem fails in general if $X$ is not irreducible;
here is a trivial example. Let $X=\{(z,w)\in \C^2\colon zw=0\}$
be the union of two complex lines. Let $u\colon X\to \R$ 
be defined by $u(z,0)=0$ for $z\in \C^*$ and
$u(0,w)=1$ for all $w\in \C$. Clearly $u$ is upper 
semicontinuous. We have $\wh u(z,0)=0$ for $z\in \C$ 
and $\wh u(0,w)=1$ for $w\in\C^*$; hence $\wh u$ 
fails to be upper semicontinuos at the point $(0,0)$. 
On the other hand, local irreducibility
is not always necessary. For example, if $X$
is the Riemann sphere with one simple double point,
then the above example does not work since the boundaries
of analytic discs through any of the two local 
branches at the double point also reach the other branch.
\qed\end{example}

%
%
%
%
%
The operator $P_u\colon \cA(\D,X) \to\R\cup \{-\infty\}$ appearing 
in (\ref{eq:Poisson-funct}),
\[
	P_u(f)= \int^{2\pi}_0 u(f(\E^{\I t}))\, \frac{\d t}{2\pi},\qquad
	f\in \cA(\D,X)
\]
is called the {\em Poisson functional}. Note that $P_u(f)$ 
is the value at $0\in \D$ of the harmonic function on 
$\D$ with boundary values $u(f(\E^{\I t}))$.
For this reason, the function $\wh u$ defined by (\ref{eq:Poisson-funct})
is also called the {\em Poisson envelope} of $u$.

Let us first justify the last statement in the theorem.
Recall that an upper semicontinuous function
$v\colon X\to\R\cup\{-\infty\}$ on a complex space 
that is not identically $-\infty$ on any irreducible 
component of $X$ is plurisubharmonic 
if every point $x\in X$ admits a neighborhood $U\subset X$,
embedded as a closed complex subvariety in a domain
$\Omega\subset \C^N$, such that $v|_U$ is the restriction 
to $U$ of a plurisubharmonic function $\wt v$ on $\Omega$. 
By \cite{FN} a function $v$ as above 
is plurisubharmonic if and only if the composition 
$v\circ f$ with any holomorphic disc $f\colon \D\to X$
is subharmonic on $\D$. This holds if and only if $v$  
satisfies the submeanvalue property on analytic discs, 
which precisely means that $v(x)\le P_v(f)$ 
for every $x\in X$ and $f\in \cA(\D,X,x)$. Since 
for the constant disc $f(\zeta)=x$ we have $P_v(f)=v(x)$,
we conclude that 

\begin{lemma}
\label{characterization}
An upper semicontinuous function $v\colon X\to\R\cup\{-\infty\}$
on a complex space $X$ that is not identically $-\infty$ 
on any irreducible component of $X$ is plurisubharmonic 
if and only if $v=\wh v$, where $\wh v$ is defined 
by (\ref{eq:Poisson-funct}).
\end{lemma}

If $v\le u$ then clearly $P_v(f)\le P_u(f)$ for every $f\in\cA(\D,X)$,
and hence $\wh v\le \wh u$. It follows that any plurisubharmonic
function $v$ for which $v\le u$ satisfies $v=\wh v\le \wh u\le u$, so 
$\wh u$ is indeed the largest such function. 

One of the main applications of Theorem \ref{Poletsky-Rosay} 
in the classical case $X=\C^n$ is the characterization of
the polynomially convex hull $\wh K$ of a compact set 
$K\subset \C^n$ by analytic discs, due to Poletsky \cite{Poletsky1991,Poletsky1993}
and Bu and Schachermayer \cite{Bu-Schachermayer}.
(See also Remark \ref{DSW} below.)
In the situation considered in this paper
we obtain a characterization of {\em plurisubharmonic hulls}
in terms of analytic discs, a fact that was already observed 
(for complex manifolds) by L\'arusson and Sigurdsson and by Rosay. 
Let $\Psh(X)$ denote the set of all plurisubharmonic functions 
on $X$. Given a compact set $K$ in $X$, its plurisubharmonic hull 
is defined by 
\[
	\wh K_{\Psh(X)} = 
	\bigl\{x\in X\colon u(x)\le \sup_K u, \ \forall u\in \Psh(X)\bigr\}.
\]
Since the modulus $|f|$ of a holomorphic function 
is a plurisubharmonic function, we always have 
$\wh K_{\Psh(X)} \subset \wh K_{\cO(X)}$. 
It is a much deeper result of Grauert \cite{Grauert:Levi}
and Narasimhan \cite{Narasimhan1961} that the two hulls 
coincide if $X$ is a Stein space or, more generally,
a 1-convex complex space. 
(See also the paper \cite{FN} and \cite{Grauert-Remmert1977}.)
In the case when $X=\C^n$, the equality of the two hulls
also follows from Poletsky's theorem as was pointed out in
\cite[Theorem 5.1]{Poletsky1996}. Related results
concerning {\em pluripolar hulls} of compact sets in 
$\C^n$ were obtained in \cite{Levenberg-Poletsky}.


\begin{corollary}
\label{Cor1}
Let $K$ be a compact set in a locally irreducible complex space $X$
such that $\wh K_{\Psh(X)}$ is compact. 
Choose an open set $V\Subset X$ containing $\wh K_{\Psh(X)}$. 
Then a point $x\in X$ belongs to $\wh K_{\Psh(X)}$ if and only
if for every open set $U\supset K$ and every number 
$\epsilon>0$ there exist a disc $f\in \cA(\D,V)$ 
and a set $E_f\subset [0,2\pi]$ of 
Lebesgue measure $|E_f| <\epsilon$ such that 
\[
	 	f(0)=x\quad {\rm and} \quad
	 	f(\E^{\I t})\in U\ \hbox{for all } t\in [0,2\pi]\bs E_f.  
\]
\end{corollary}

\begin{proof}
Assume first that a point $x\in X$ satisfies the stated conditions;
we shall prove that $x\in \wh K_{\Psh(X)}$.
Choose a function $\rho\in \Psh(X)$. Set $M=\sup_K \rho$ and $M'=\sup_{V} \rho$. 
Pick a number $\epsilon>0$ and an open set $U$ with 
$K\subset U\Subset V$  such that $\sup_U \rho < M+\epsilon$. 
Let the disc $f\in \cA(\D,V,x)$ and the set $E_f\subset [0,2\pi]$ 
satisfy the hypotheses of the corollary. Then 
\[
	\rho(x) \le P_\rho(f)= \int_{E_f} \rho(f(\E^{\I t}))\, \frac{\d t}{2\pi} 
	          + \int_{[0,2\pi]\bs E_f} \rho(f(\E^{\I t}))\, \frac{\d t}{2\pi} 
          < M'\epsilon + M+\epsilon.  
\]
Since this holds for every $\epsilon>0$, we get
$\rho(x)\le M$. As $\rho\in \Psh(X)$ was arbitrary,
we conclude that  $x\in \wh K_{\Psh(X)}$.

Conversely, assume that $x\in \wh K_{\Psh(X)}$. 
The function $u\colon V\to\R$ which equals $-1$ on the 
open set $U\supset K$ and equals $0$ on $V\bs U$ 
is upper semicontinuous. 
Let $v=\wh u$ be the associated plurisubharmonic function
defined by (\ref{eq:Poisson-funct}). 
Then  $-1\le v\le 0$ on $V$,
and $v(x)=-1$. Theorem \ref{Poletsky-Rosay} furnishes
a disc $f\in \cA(\D,V,x)$ with $P_u(f) < -1+ \epsilon/2\pi$. 
By the definition of $u$ this implies that the set 
$E_f=\{t\in [0,2\pi] \colon f(\E^{\I t}) \notin U\}$ has measure 
at most $\epsilon$.
\end{proof}

\begin{remark}
\label{DSW}
Let $K$ be a compact subset of $\C^n$.
It was shown by Duval and Sibony \cite{Duval-Sibony1,Duval-Sibony2} 
that for any point $p\in\wh K$ and Jensen measure $\sigma$
representing $p$ there exists a positive current $T$ 
of bidimension $(1,1)$ such that $\d\d^c T =\sigma-\delta_p$, where $\delta_p$ 
denotes the point evaluation at $p$. Recently Wold \cite{Wold2010} showed that every  
Duval-Sibony current $T$ is a weak limit of currents $T_j=(f_j)_* G$, 
where $G$ is the {\em Green current} on the unit disc $\D$,
given by 
\[
	G(\omega)= - \int_\D  \log |\zeta| \cdotp \omega,\qquad
	\omega\in \mathcal{E}^{1,1}(\clD),
\]
and $f_j\colon \clD\to\C^n$ is a sequence of Poletsky discs.
On the other hand, it has been known since the classical
examples of Stolzenberg and Alexander that the polynomial 
hull of $K$ can not be explained in general by 
analytic varieties with boundaries in $K$.
(In this direction see the recent paper of Dujardin \cite{Dujardin}.)
Hence Poletsky's characterization of the polynomial 
hull remains the most universal one that we have at the moment.
For more about hulls we refer the interested reader to Stout's monographs
\cite{Stout:uniformalgebras,Stout:polynomial}.
\qed\end{remark}


\begin{remark}
Another immediate implication of Theorem \ref{Poletsky-Rosay} 
is the following result which was observed in the smooth case by 
J.-P.\ Rosay \cite[Corollary 0.2]{Rosay:Poletsky1}:
If $X$ is as in Theorem \ref{Poletsky-Rosay} and if 
every bounded plurisubharmonic function on $X$ is constant, then 
for every point $p\in X$, nonempty open set $U\subset X$ 
and number $\epsilon>0$ there exists an analytic disc 
$f\colon \ol \D\to X$ such that $f(0)=p$ and 
the set $\{ t \in [0,2\pi)\colon f(\E^{\I t})\in U\}$
has measure at least $2\pi-\epsilon$. This follows
by observing that the envelope $\wh u$ defined by (\ref{eq:Poisson-funct}) of the 
negative characteristic function $u=-\chi_U$ of the set $U$ is bounded
from above by $0$, and hence it is constantly equal to $-1$.

We expect that the methods of our proof can be used 
to extend the main theorem of Rosay in \cite{Rosay:Poletsky3} 
as follows: A locally irreducible complex space $X$
does not admit any nonconstant bounded plurisubharmonic function
(such space is said to be Liouville) if and only if
every closed loop in $X$ can be approximated on a
set of almost full linear measure in the circle 
by the boundary values of holomorphic discs in $X$.
We hope to return to these questions in a future publication.
\qed\end{remark}


\section{A nonlinear Cousin-I problem} 
\label{sprays}
We recall from \cite{BDF1,FFAsian} 
the relevant results concerning holomorphic sprays,
adjusting them to the applications in this paper.

\begin{definition}
\label{Spray}
Let $\ell\ge 2$ and $r\in\{0,\ldots,\ell\}$ be integers.
Assume that $X$ is a complex space, $D$ is a relatively compact 
domain with $\cC^\ell$ boundary in $\C$,
and $\sigma$ is a finite set of points in $D$.
A {\em spray of maps} of class $\cA^r(D)$ with the 
exceptional set $\sigma$ and with values in $X$ 
is a map $f\colon P\times \bar D\to X$, where $P$ 
(the {\em parameter set} of the spray) is an open subset of 
a Euclidean space $\C^m$ containing the origin, 
such that the following hold: 
\begin{itemize}
\item[(i)]   $f$ is holomorphic on $P\times D$ and of class $\cC^r$
on $P\times \bar D$, 
\item[(ii)] the maps $f(0,\cdotp)$ and $f(t,\cdotp)$ agree 
on $\sigma$ for all $t\in P$, and
\item[(iii)] if $z\in \bar D\bs \sigma$ and $t\in P$ 
then $f(t,z)\in X_{\reg}$ and 
\[
	\di_t f(t,z) \colon T_t \C^m \to T_{f(t,z)} X
\]
is surjective (the {\em domination property}).
\end{itemize}
We call $f_0=f(0,\cdotp)$ the {\em core} 
(or {\em central}\/) map of the spray $f$. 
\end{definition}

The following lemma is a special case of 
\cite[Lemma 4.2]{BDF1}.

\begin{lemma}
\label{sprays-exist}
{\em (Existence of sprays)} 
Assume that $\ell$, $r$, $D$, $\sigma$ and $X$ are in
accordance with Definition \ref{Spray}.  Given a map 
$f_0\colon \bar D\to X$ of class $\cA^r(D)$ such that the set
$\{z\in \bar D \colon f_0(z) \in X_{\sing}\}$ 
is contained in $\sigma$,
there exists a spray $f\colon P\times \bar D\to X$ of class $\cA^r(D)$, 
with the exceptional set $\sigma$, such that $f(0,\cdotp)=f_0$.  
\end{lemma}

\begin{definition}
\label{Cartan-pair}
Let $\ell\ge 2$ be an integer.
A pair of open subsets $D_0,D_1 \Subset \C$ 
is said to be a {\em Cartan pair of class $\cC^\ell$} if 
\begin{itemize}
\item[(i)]  $D_0$, $D_1$, $D=D_0\cup D_1$ and $D_{0,1}=D_0\cap D_1$ 
are domains with $\cC^\ell$ smooth boundaries, and 
\item[(ii)] 
$\overline {D_0\backslash D_1} \cap \overline {D_1\backslash D_0}=\emptyset$ 
(the separation property). 
\end{itemize}
\end{definition}

The following is the main result on gluing sprays in the particular
situation that we are considering 
(see \cite[Proposition 4.3]{BDF1} or \cite[Lemma 3.2]{FFAsian}).
This is in fact a solution of a nonlinear Cousin-I problem.


\begin{proposition}
\label{gluing-sprays}
{\em (Gluing sprays)} 
Let $(D_0,D_1)$ be a Cartan pair of class $\cC^\ell$ 
$(\ell\ge 2)$ in $\C$ (Def.\ \ref{Cartan-pair}). 
Set $D=D_0\cup D_1$, $D_{0,1}=D_0\cap D_1$. Let
$X$ be a complex space. Given an integer $r\in\{0,1,\ldots, \ell\}$  
and a spray $f\colon P_0\times \bar D_0 \to X$ of class $\cA^r(D_0)$ 
with the exceptional set $\sigma$ such that 
$\sigma \cap \bar D_{0,1}=\emptyset$, there is an open 
set $P \ss P_0\subset \C^m$ containing $0\in\C^m$ 
and satisfying the following. 

For every spray $f' \colon P_0\times \bar D_1 \to X$
of class $\cA^r(D_1)$, with the exceptional set $\sigma'$
such that $f'$ is sufficiently $\cC^r$ close to $f$ on 
$P_0\times \bar D_{0,1}$ and $\sigma'\cap \bar D_{0,1}=\emptyset$, 
there exists a spray $F \colon P\times \bar D \to X$ 
of class $\cA^r(D)$, with the exceptional set 
$\sigma\cup \sigma'$, enjoying the following properties:
\begin{itemize}
\item[(i)] the restriction $F\colon P\times \bar D_0\to X$ 
is close to $f \colon P\times \bar D_0 \to X$ in the $\cC^r$-topology
(depending on the $\cC^r$-distance of $f$ and $f'$ on $P_0\times \bar D_{0,1}$),
\item[(ii)] the core map $F_0=F(0,\cdotp)$ is homotopic to 
$f_0=f(0,\cdotp)$ on $\bar D_0$, 
and $F_0$ is homotopic to $f'_0=f'(0,\cdotp)$ on $\bar D_1$, 
\item[(iii)] $F_0$ agrees with $f_0$ on $\sigma$, 
and it agrees with $f'_0$  on $\sigma'$, and 
\item[\rm (iv)] 
$F(t,z)\in\{f'(s,z)\colon s \in P_0\}$ for each $t\in P$ and
$z\in \bar D_1$.
\end{itemize}
\end{proposition}

Here is a brief outline of the proof. 
The first step is to find a domain $P'\subset \C^m$ such that 
$0\in P'\Subset P_0$, and a transition map 
between  the sprays $f$ and $f'$, that is, a $\cC^r$ map 
\[
	\gamma\colon P'\times \bar D_{0,1} \to P_0 \times \bar D_{0,1},
	\quad \gamma(t,z)=\bigl(c(t,z),z\bigr)
\]
that is holomorphic in $P'\times D_{0,1}$ and is $\cC^r$ close to the 
identity map (depending on the $\cC^r$-distance between
$f$ and $f'$ over $P_0\times \bar D_{0,1}$), such that 
\[
	f=f'\circ \gamma \quad{\rm  on}\ P'\times \bar D_{0,1}.
\]
This is an application of both the implicit function theorem
and the fact that Cartan's Theorem B holds for holomorphic
vector bundles on domains in $\C$ that are 
smooth of class $\cC^r$ up to the boundary.
The key step is to split the map $\gamma$ in the form
\[
	\gamma= \beta \circ \alpha^{-1}
\]
where $\alpha(t,z)=\bigl(a(t,z),z\bigr)$ and 
$\beta(t,z)=\bigl(b(t,z),z\bigr)$ are maps with similar
properties over $P\times \bar D_0$ and $P\times \bar D_1$,
respectively, for some slightly smaller parameter set 
$0\in P\ss P'$. This splitting is accomplished by nonlinear 
operators whose linearization involves a solution operator 
for the $\dibar$-equation with $\cC^r$ estimates
on $D=D_0\cup D_1$. The final step is to observe that
over $P\times \bar D_{0,1}$ we have 
\[
	f=f'\circ\gamma =f'\circ\beta\circ\alpha^{-1}
	\Longrightarrow  f\circ \alpha=f'\circ \beta.
\]
Hence the two sides amalgamate into a spray $F$ over $\bar D$,
and it is easily verified that $F$ satisfies Proposition
\ref{gluing-sprays}.

%
\section{A Riemann-Hilbert problem} 
\label{RH}
In this section we explain how to find an approximate solution 
of a Riemann-Hilbert problem with the control 
of the average of a given function on a boundary arc.
Results of this kind have been used by several authors; see e.g.\ 
\cite{Poletsky1991, Poletsky1993,Bu-Schachermayer, FGlobevnik1, FGlobevnik2}.

Recall that $\bT=b\,\D=\{\zeta\in\C\colon |\zeta|=1\}$. 
Given a measurable subset $I\subset\bT$ and a measurable function
$v\colon I\to \R$, $\int_I v(\E^{\I t})\,\d t$ 
will denote the integral over the set of points 
$t\in[0,2\pi]$ for which $\E^{\I t}\in I$.

\begin{lemma}
\label{RH1} 
Let $f\in\cA(\D,\C^n)$, and let $g\colon \bT\times \clD \to \C^n$ 
be a continuous map such that for each $\zeta\in \bT$ 
we have $g(\zeta,\cdotp)\in \cA(\D,X,f(\zeta))$.
Given numbers $\epsilon>0$ and $0<r<1$, an arc $I\subset \bT$,
and a continuous function $u\colon\C^n\to\R$, 
there are a number $r'\in [r,1)$ 
and a disc $h\in \cA(\D,\C^n,f(0))$ satisfying 
\begin{equation}
\label{small-increase}
	\int_{I}
    u\bigl( h(\E^{\I t})\bigr) 
    \, \frac{\d t}{2\pi} < 
	  \int^{2\pi}_0 \!\! \int_{I} 
       u\bigl(g(\E^{\I t},\E^{\I\theta})\bigr) 
       \frac{\d t}{2\pi} \frac{\d\theta}{2\pi} + \epsilon
\end{equation}
and also the following properties:
\begin{itemize}
\item[\rm (i)] for any $\zeta\in \bT$ we have 
$\dist \bigl( h(\zeta),g(\zeta,\bT)\bigr) < \epsilon$, 
\item[\rm (ii)] for any $\zeta\in \bT$ and $\rho\in [r',1]$ 
we have $\dist \bigl( h(\rho \zeta),g(\zeta,\clD)\bigr) < \epsilon$,
\item[\rm (iii)] for any $|\zeta|\le r'$ we have 
$|h(\zeta)-f(\zeta)| < \epsilon$, and
\item[\rm (iv)]
if $g(\zeta,\cdotp)=f(\zeta)$ is the constant disc
for all $\zeta\in \bT\setminus J$, where $J\subset \bT$
is an arc containing $\ol I$, then we can choose $h$ 
such that $|h-f|<\epsilon$ holds outside any given neighborhood 
of $J$ in $\clD$. 
\end{itemize}
\end{lemma}

\begin{proof} 
Write 
\[
	g(\zeta,z)=f(\zeta)+\lambda(\zeta,z), \qquad \zeta\in\bT,\ z\in \clD,
\]
where $\lambda(\zeta,z)$ is continuous on 
$(\zeta,z)\in \bT\times \clD$, and for every fixed 
$\zeta\in \bT$ the function $\clD \ni z\mapsto \lambda(\zeta,z)$ 
is holomorphic on $\D$ and satisfies $\lambda(\zeta,0)=0$.
We can approximate $\lambda$ uniformly on $\bT\times \clD$ 
by Laurent polynomials of the form
\[
	\wt\lambda(\zeta,z) = \frac{1}{\zeta^m} \sum_{j=1}^N A_j(\zeta) z^j
	= \frac{z}{\zeta^m} \sum_{j=1}^N A_j(\zeta) z^{j-1}
\]
with polynomial coefficients $A_j(\zeta)$.
Hence we can choose a map $\wt \lambda$ as above
and a number $r'\in [r,1)$ such that
\begin{equation}
\label{est:lambda}
	 \big| \wt\lambda(\rho\,\E^{\I t},z) - \lambda(\E^{\I t},z)\big| 
	 < \frac{\epsilon}{2},
	 \qquad t\in\R,\ r'\le \rho \le 1,\ |z|\le 1
\end{equation}
and
\begin{equation}
\label{est:f}
	\big| f(\rho \E^{\I t}) - f(\E^{\I t}) \big| < \frac{\epsilon}{2},
	 \qquad t\in\R,\ r'\le \rho \le 1.
\end{equation}
Choose an integer $k>m$ and a number $c=\E^{\I \phi} \in \bT$ and set 
\[
	h_{k}(\zeta,c) = f(\zeta) + \wt\lambda(\zeta, c\zeta^k)
	          =f(\zeta) + c\, \zeta^{k-m} 
	          \sum_{j=1}^N A_j(\zeta) \left(c\zeta^k\right)^{j-1},
	          \quad |\zeta|\le 1.
\]
This is an analytic disc in $\C^n$ satisfying  
$h_{k}(0,c)=f(0)$, so it belongs to $\cA(\D,\C^n,f(0))$. 
For $\zeta=\E^{\I t}\in \bT$ we have
\begin{equation}
\label{eq:h-closeto-g}
	h_k\bigl(\E^{\I t},c\bigr) = 	f\bigl(\E^{\I t}\bigr) + 
	\wt\lambda\bigl(\E^{\I t}, \E^{\I\phi} \E^{k\I t}\bigr)
	\approx g\bigl(\E^{\I t},\E^{\I(\phi+ kt)}\bigr),
\end{equation} 
and hence property (i) holds in view of (\ref{est:lambda}).
Similarly, if $r'\le \rho \le 1$ then 
\begin{eqnarray}
	\left| h_k\bigl( \rho \E^{\I t},c \bigr) - 
	g\bigl(\E^{\I t}, c \rho^k \E^{\I kt}\bigr) \right|  
	&\le& 
	\left| \wt\lambda\bigl(\rho \E^{\I t}, c\,\rho^k \E^{\I kt}\bigr) 
	- \lambda\bigl(\E^{\I t}, c\,\rho^k \E^{\I kt}\bigr) \right| \cr
	&& + \ \big| f(\rho \E^{\I t}) - f(\E^{\I t}) \big|  \cr
	&<& \epsilon
\end{eqnarray}
by (\ref{est:lambda}) and (\ref{est:f}), so (ii) holds as well.

If $k\to +\infty$ then $h_k(\zeta,\cdotp)\to f(\zeta)$ 
uniformly on the set $\{|\zeta|\le r'\}\times \bT$, so we  
get condition (iii) if $k$ is chosen big enough.
Property (iv) is a consequence of (ii) and (iii).

It remains to show that the inequality (\ref{small-increase})
can be achieved by a suitable choice of the 
number $c=\E^{\I \phi}\in \bT$. We clearly have
\[
	\int^{2\pi}_0 \!\!\! \int_{I}
       u\bigl(g(\E^{\I t},\E^{\I\theta})\bigr) 
       \,\frac{\d t}{2\pi} \frac{\d\theta}{2\pi} 
  	= \int^{2\pi}_0 \!\!\! \int_I
       u\bigl(g(\E^{\I t},\E^{\I(\theta + kt)})\bigr)
       \, \frac{\d t}{2\pi} \frac{\d\theta}{2\pi}.
\]
By the mean value theorem there exists $\phi\in [0,2\pi)$
such that this equals 
\[
	\int_I u\bigl(g(\E^{\I t},\E^{\I(\phi + kt)}) \bigr)
	\,\frac{\d t}{2\pi}.
\]
By (\ref{eq:h-closeto-g}) this number differs by at most 
$\epsilon$ from 
$
\int^{2\pi}_0 u\bigl(h_k(\E^{\I t},\E^{\I \phi})\bigr) 
\,\frac{\d t}{2\pi}
$ 
if $k$ is chosen big enough. This completes the proof.
\end{proof}

%
\section{Proof of Theorem \ref{Poletsky-Rosay}}

{\bf Step 1.}
We reduce to the case when $X$ is a normal complex space. 

Recall that a function on a complex space $X$ is {\em weakly holomorphic} 
if it is holomorphic on the regular part $X_{\reg}$ and is 
locally bounded near each singular point.
A reduced complex space $X$ is said to be {\em normal}
if every weakly holomorphic function is in fact holomorphic;
it is {\em seminormal} if 
every continuous weakly holomorphic function on $X$ is holomorphic.
A holomorphic map $\pi\colon\widetilde X\to X$ of complex spaces
is called a {\em seminormalization} (or a  {\em maximalization}) 
of $X$ if $\widetilde X$ is a seminormal complex space and $\pi$
is a homeomorphism. We will use the following facts  
(see \cite{Remmert}):
\begin{itemize}
\item every reduced complex space admits a seminormalization,
\item every locally irreducible seminormal complex space is normal, and 
\item seminormalization is a functor; in particular, 
the lift $\pi^{-1}\circ f$ of any holomorphic disc $f\colon \D\to X$ 
is a holomorphic disc in $\widetilde X$.
\end{itemize}
This implies that if $\pi\colon \wt X\to X$ is a seminormalization
of $X$ then the composition $u\mapsto u\circ \pi$ 
induces an isomorphism from $PSH(X)$ onto $PSH(\widetilde X)$.

Assume now that Theorem \ref{Poletsky-Rosay} holds for normal complex spaces.
Let $X$ be a locally irreducible complex space and
$u\colon X\to \R\cup\{-\infty\}$ an upper semicontinuous function. 
Let $\pi\colon\widetilde X\to X$ be a seminormalization of $X$.
The function $\widetilde u=u\circ \pi$ is upper semicontinuous on $\widetilde X$. 
Since $\widetilde X$ is a normal complex space, 
the Poisson envelope $\widetilde v = \wh{\widetilde u}$ of $\wt u$
is plurisubharmonic on $\widetilde X$.
Each holomorphic disc $f\in \cA(\D,X)$ admits a holomorphic lifting 
$\widetilde f\in \cA(\D,\widetilde X)$ with $\pi\circ\widetilde f=f$,
and we clearly have $P_{\widetilde u}(\widetilde f)=P_u(f)$. 
This implies that $\widetilde v(\pi^{-1}(x))=\hat (x)$ 
and hence $\hat u$ is plurisubharmonic.

\smallskip
{\bf Step 2.}
We reduce to the case when the function 
$u\colon X\to\R$ is continuous and bounded from below.

Since $u\colon X\to \R\cup\{-\infty\}$ 
is upper semicontinuous, there is a decreasing sequence 
of continuous functions $u_1\ge u_2\ge \ldots\ge u$ 
such that $u=\lim_{k\to\infty} u_k$ pointwise in $X$. 
Replacing $u_k$ by $\max\{u_k,-k\}$ we may assume 
in addition that $u_k\ge -k$ on $X$; hence $\wh u_k\ge -k$ as well.
Assuming that the result holds for each $u_k$, 
the Poisson envelopes $\wh u_1\ge\wh u_2\ge \ldots$ 
form a decreasing sequence of plurisubharmonic functions,
and hence  $v=\lim_{k\to\infty} \wh u_k$
is also plurisubharmonic or identically $-\infty$.
Since $\wh u_k\le u_k$, it follows that $v\le u$.
For any point $x\in X$ and disc $f\in \cA(\D,X,x)$ 
we have  $v(x)\le P_v(f) \le P_u(f)\le P_{u_k}(f)$
by monotonicity. Taking the infimum over all such $f$ shows that
$v(x)\le \wh u(x)\le \wh u_k(x)$. Letting $k\to\infty$
gives $v(x)=\wh u(x)$, so $\wh u$ is plurisubharmonic
or $-\infty$.

\smallskip
{\bf Step 3.}
We show that the Poisson envelope $v=\wh u$ 
given by (\ref{eq:Poisson-funct}) is upper semicontinuous
on the regular locus $X_{\reg}$.
(We assume that $X$ is a reduced complex space. The reductions 
in Steps 1 and 2 will not be used here.)

Pick a point $x\in X_{\reg}$ and a number $\epsilon>0$. 
Assume first that $v(x)>-\infty$.
By the definition of $v$ there exists a disc 
$f_0\in \cA(\D,X,x)$ such that $v(x)\le P_u(f_0) < v(x)+\epsilon$.
By shrinking $\D$ slightly we may assume that $f_0(\bT)\subset X_{\reg}$;
hence the set $\sigma=\{\zeta\in\D \colon f_0(\zeta)\in X_{\sing}\}$
is finite and $0\notin\sigma$. 
We embed $f_0$ as the central map $f_0=f(0,\cdotp)$ in a 
spray of holomorphic discs $f\colon  P\times\clD\to X$ 
with the exceptional set $\sigma$, 
where $P$ is an open set in $\C^m$ containing the origin. 
(See Def.\ \ref{Spray} and Lemma \ref{sprays-exist};
on $X=\C^n$ we can simply use the family of translates $f_y=f_0+(y-x)$).
If $P'\Subset P$ is a small open neighborhood of $0\in \C^m$ then 
$P_u(f(t,\cdotp))< P_u(f_0)+\epsilon$ for each $t\in P'$ and hence 
\[
	v(f(t,0))\le P_u(f(t,\cdotp)) \le  P_u(f_0)+\epsilon < v(x)+2\epsilon.
\]
By the domination property the set 
$\{f(t,0)\colon t\in P'\}$ fills a neighborhood of the point 
$x=f(0)$ in $X$, so we see that $v$ is upper semicontinuous at $x$.
A similar argument works at points where $v(x)=-\infty$.

\smallskip
{\bf Step 4.}
In this main step of the proof we assume that 
$X$ is an irreducible normal complex space and that
$u\colon X\to\R$ is a continuous function which is 
bounded from below (see Steps 1 and 2). We shall 
prove that the Poisson envelope $v=\wh u$ given 
by (\ref{eq:Poisson-funct}) is plurisubharmonic on $X_{\reg}$. 
(Plurisubharmonicity on $X_{\sing}$ will be proved in Step 5.)

We need to show that for every point 
$x\in X_{\reg}$ and for every analytic disc 
$f\in \cA(\D,X,x)$  we have the submeanvalue property
\begin{equation}
\label{eq:submean}
	v(x) = v(f(0)) \le \int^{2\pi}_0 v(f(\E^{\I t}))\, \frac{\d t}{2\pi}.
\end{equation}
Since plurisubharmonicity is a local property, it is enough to 
consider small discs; hence we may assume that $f(\clD)\subset X_{\reg}$ and 
that $f$ is holomorphic on a larger disc $r_0\D$ for some $r_0>1$.

Pick a number $\epsilon>0$. Fix a point $\E^{\I t} \in \bT$.
By the definition of $v$ there exists an analytic disc 
$g_t = g(\E^{\I t},\cdotp) \in \cA(\D,X,f(\E^{\I t}))$ 
such that
\begin{equation}
\label{eq:suboptimal}
	v(f(\E^{\I t}))\le 
	\int_0^{2\pi} u\bigl(g(\E^{\I t},\E^{\I \theta})\bigr)\, \frac{\d \theta}{2\pi} 
	 < v(f(\E^{\I t})) +\epsilon.
\end{equation}
Since the set $\{\zeta \in \D\colon g_t(\zeta)\in X_{\sing}\}$ is discrete,
we can replace $g_t$ by the map $\zeta\mapsto g_t(r\zeta)$ for some 
$r<1$ very close to $1$ so that the new map still satisfies
(\ref{eq:suboptimal}) and $g_t(\bT)\subset X_{\reg}$. 
By Lemma \ref{sprays-exist} there is a domain $P\subset \C^m$
containing the origin and a holomorphic spray of discs 
$G\colon P\times \clD \to X$ with the
central map $G(0,\cdotp)=g_t$. 
(We use sprays of class $\cA^0(\D)$.)
Since $G(0,0)=g_t(0)=f(\E^{\I t})\in X_{\reg}$ 
and the spray $G$ is dominating, the set 
$G(P,0)=\{G(w,0)\colon w\in P\}$ 
is a neighborhood of the point $f(\E^{\I t})$ in $X$.
By the implicit mapping theorem there are a 
disc $D\subset r_0\D$ centered at the point $\E^{\I t}\in \bT$ 
and a holomorphic map $\varphi\colon D \to P$ such that 
\[	
	\varphi(\E^{\I t})= 0 \quad {\rm and}\quad 
	G(\varphi(\zeta),0)=f(\zeta),\quad \zeta\in D.
\]
Consider the continuous map $g\colon D\times \clD\to X$ defined by  
\[
	g(\zeta,z)=G(\varphi(\zeta),z), \qquad \zeta\in D,\ z\in\clD. 
\]
Note that $g$ is  holomorphic on $D\times\D$, and 
\[
   g(\E^{\I t},\cdotp)=G(0,\cdotp)=g_t; \qquad 
   g(\zeta,0)= f(\zeta),\quad \zeta\in D.
\]
Since $g(\zeta,\cdotp)$ is uniformly close to 
$g_t$ when $\zeta$ is close to $\E^{\I t}$,
it follows from (\ref{eq:suboptimal}) that 
there is a small arc $I\Subset \bT\cap D$ 
around the point $\E^{\I t}$ such that
\[
      \int^{2\pi}_0 \!\! \int_{I} 
       u\bigl(g(\E^{\I \eta},\E^{\I \theta})\bigr) \, 
       \frac{\d\eta}{2\pi} \frac{\d \theta}{2\pi} 
       \le 
       \int_I v\bigl(f(\E^{\I \eta})\bigr)\,\frac{\d\eta}{2\pi} + 
       \frac{|I|\,\epsilon}{2\pi}.     
\]

By repeating this construction at other points of $\bT$
we find finitely many pairs of arcs 
$I_j\ss I'_j \ss \bT\cap D_j$ $(j=1,\ldots,l)$, 
where $D_j$ is a disc contained in $r_0\D$, 
such that $\ol I'_j\cap \ol I'_k=\emptyset$ if $j\ne k$ and
the set $E=\bT\bs \cup_{j=1}^m I_j$ has 
arbitrarily small measure $|E|$, and 
holomorphic families of discs $g_j(\zeta,z)$ for $\zeta\in D_j$
and $z\in\clD$ such that 
\begin{equation}
\label{est1}
   \int^{2\pi}_0 \!\! \int_{I_j} 
       u\bigl(g_j(\E^{\I t},\E^{\I\theta})\bigr) \, 
      \frac{\d t}{2\pi} \frac{\d\theta}{2\pi} 
   < \int_{I_j} v\bigl(f(\E^{\I t})\bigr)\,
   \frac{\d t}{2\pi}
    + \frac{|I_j|\,\epsilon}{2\pi}.
\end{equation}

For each $j=1,\ldots,l$ we choose a smoothly bounded simply 
connected domain $\Delta_j \subset D_j\cap \D$ 
such that $\ol\Delta_j \subset D$, 
$\ol\Delta_j\cap \bT=\ol{I'_j}$, and  
$\ol \Delta_j\cap \ol\Delta_k=\emptyset$ when  
$1\le j\ne k\le l$. (The situation with $l=3$ 
is illustrated in Fig.\ \ref{Fig1}. The reader 
should keep in mind that the gaps between the segments $I_j$ have very small total length.
The role of the sets $D_0,D_1\subset \D$ with
$D_0\cup D_1=\D$ is explained in the proof of Lemma \ref{RH2} below.)

%
%
%
%
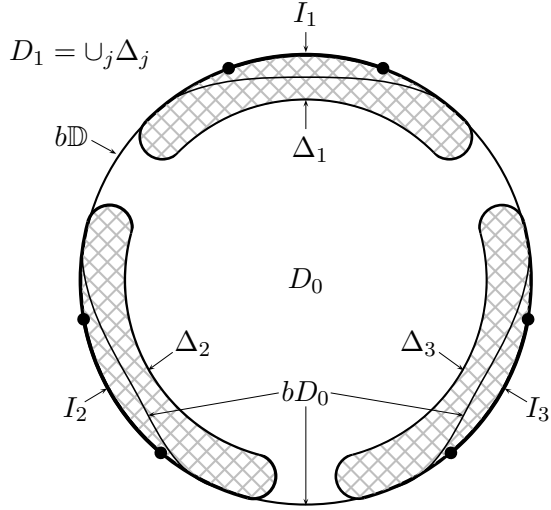
\begin{figure}[ht]
\psset{unit=0.6cm, linewidth=0.6pt}  
\begin{pspicture}(-5.5,-6)(6,6)

\pscircle[linewidth=0.8pt](0,0){5}

%
%
%
%
\pscustom[linestyle=none,fillstyle=crosshatch,hatchcolor=lightgray,linewidth=0pt]
{
\psarc(0,0){5}{45}{135}
\psarc(-3.18,3.18){0.5}{135}{315}
\psarc(0,0){4}{45}{135}
\psarc(3.18,3.18){0.5}{-135}{45}
}

%
%
\psarc[linewidth=1pt](0,0){5}{45}{135}
\psarc[linewidth=1pt](0,0){4}{45}{135}
\psarc[linewidth=1pt](-3.18,3.18){0.5}{135}{315}
\psarc[linewidth=1pt](3.18,3.18){0.5}{-135}{45}

%
%
\psarc[linewidth=1.4pt,arrows=*-*](0,0){5}{70}{110}
\psecurve(-4,3)(-3,4)(0,4.5)(3,4)(4,3)

%
%
\pscustom[linestyle=none,fillstyle=crosshatch,hatchcolor=lightgray]
{
\psarc(0,0){5}{165}{255}
\psarc(-1.165,-4.347){0.5}{255}{75}
\psarc(0,0){4}{165}{255}
\psarc(-4.347,1.165){0.5}{-15}{165}
}

\psarc[linewidth=1.2pt](0,0){5}{165}{255}
\psarc[linewidth=1.2pt](0,0){4}{165}{255}
\psarc[linewidth=1.2pt](-4.347,1.165){0.5}{-15}{165}
\psarc[linewidth=1.2pt](-1.165,-4.347){0.5}{255}{75}

\psarc[linewidth=1.4pt,arrows=*-*](0,0){5}{190}{230}
\psecurve(4.598,-1.96)(-1.964,-4.5981)(-3.897,-2.25)(-4.964,0.598)(-4.598,1.964)

%
%
\pscustom[linestyle=none,fillstyle=crosshatch,hatchcolor=lightgray]
{
\psarc(0,0){5}{285}{15}
\psarc(4.347,1.165){0.5}{15}{195}
\psarc(0,0){4}{285}{15}
\psarc(1.165,-4.347){0.5}{105}{285}
}

\psarc[linewidth=1pt](0,0){5}{285}{15}
\psarc[linewidth=1pt](0,0){4}{285}{15}
\psarc[linewidth=1pt](4.347,1.165){0.5}{15}{195}
\psarc[linewidth=1pt](1.165,-4.347){0.5}{105}{285}

\psarc[linewidth=1.4pt,arrows=*-*](0,0){5}{310}{350}
\psecurve(-4.598,-1.96)(1.964,-4.5981)(3.897,-2.25)(4.964,0.598)(4.598,1.964)

%
%
\pscircle[linestyle=none,fillstyle=solid,fillcolor=white](0,0){4}

\rput(0,0){$D_0$}

\rput(0.1,2.9){$\Delta_1$}
\psline[linewidth=0.2pt]{->}(0,3.35)(0,4) 

\rput(-2.5,-1.4){$\Delta_2$}
\psline[linewidth=0.2pt]{->}(-2.95,-1.6)(-3.5,-2)

\rput(2.5,-1.4){$\Delta_3$}
\psline[linewidth=0.2pt]{->}(2.95,-1.6)(3.5,-2)

\rput(0,5.9){$I_1$}
\psline[linewidth=0.2pt]{->}(0,5.5)(0,5.05)

\rput(-5.1,-2.85){$I_2$}
\psline[linewidth=0.2pt]{->}(-4.9,-2.7)(-4.4,-2.4)

\rput(5.15,-2.85){$I_3$}
\psline[linewidth=0.2pt]{->}(4.9,-2.7)(4.4,-2.4)

\rput(-5.2,3.2){$b\D$}
\psline[linewidth=0.2pt]{->}(-4.75,3.1)(-4.2,2.8)

\rput(0,-2.5){$bD_0$}
\psline[linewidth=0.2pt]{->}(-0.6,-2.5)(-3.5,-3)
\psline[linewidth=0.2pt]{->}(0.55,-2.5)(3.5,-3)
\psline[linewidth=0.2pt]{->}(0,-2.8)(0,-5)

\rput(-5,5){$D_1=\cup_j \Delta_j$}

\end{pspicture}
\caption{The Cartan pair $(D_0,D_1)$}
\label{Fig1}
\end{figure}

Let $\chi\colon \C\to [0,1]$ be a smooth function
such that $\chi=1$ on $\cup_{j=1}^l I_j\subset \bT$ 
and $\chi=0$ on a neighborhood of the set 
$\clD \setminus \cup_{j=1}^l (\Delta_j\cup I'_j)$. 
Consider the map $\xi\colon\clD \times \clD\to X$ 
defined by
\begin{equation}
\label{discs-xi}
	\xi(\zeta,z)= 
		\begin{cases}  
					g_j\bigl(\zeta,\chi(\zeta)z\bigr), 
					  & \zeta \in \ol\Delta_j,\ z\in \clD,\ j=1,\ldots,l; \cr
  				f(\zeta),  & \chi(\zeta)=0,\ z\in \clD. 
  \end{cases}
\end{equation}
The latter condition holds in particular if 
$\zeta\in \bT\setminus \cup_{j=1}^m I'_j$.  
Note that $\xi$ is continuous and is holomorphic 
in the second variable. Then 
\begin{equation}
\label{est2}
   \int_0^{2\pi} \int^{2\pi}_0
       u\bigl(\xi(\E^{\I t},\E^{\I\theta})\bigr) \, 
       \frac{\d t}{2\pi} \frac{\d\theta}{2\pi}
   < \int_0^{2\pi}  v\bigl(f(\E^{\I t})\bigr)\,\frac{\d t}{2\pi}
    + 2\epsilon.
\end{equation}
Indeed, the integral over $\cup_{j=1}^l I_j$
is estimated by adding up the inequalities (\ref{est1}),
while the integral over the complementary set 
$E=\bT\setminus \cup_{j=1}^l I_j$ 
gives at most $\epsilon$ if the measure $|E|$ is small enough.  
(When estimating the integral over $E$ it is important to observe 
that $u\ge -M$ by the assumption, and $u$ is bounded from above
on each compact set. Since $-M\le v=\wh u \le u$, 
the same holds for $v$.)

To conclude the proof we apply the following lemma
to the data that we have just constructed.
We state it in a more general form since we shall
need it again in Step 5 below.

\begin{lemma}
\label{RH2} 
Let $X$ be a reduced complex space and let $u\colon X\to\R$ 
be a continuous function on $X$.
Assume that $f\in\cA(\D,X)$ is a holomorphic disc such that $f(\bT)\subset X_{\reg}$, 
and $\xi \colon \bT\times \overline \D\to X$ is a continuous map such that 
$\xi(\zeta,\cdotp)\in \cA(\D,X,f(\zeta))$ for every $\zeta\in \bT$.
Given $\epsilon>0$ there exists an analytic disc 
$h\in \cA(\D,X,f(0))$ such that 
\begin{eqnarray}
\label{small-increase2}
	\int_0^{2\pi} u(h(\E^{\I t})) \, \frac{\d t}{2\pi} < 
	\int^{2\pi}_0 \!\!\! \int^{2\pi}_0
       u\bigl(\xi(\E^{\I t},\E^{\I\theta})\bigr) \, 
       \frac{\d t}{2\pi} \frac{\d\theta}{2\pi} + \epsilon.
\end{eqnarray}
\end{lemma}

Note that the estimate (\ref{small-increase2})
is the same as (\ref{small-increase})
in Lemma \ref{RH1} which pertains to the case $X=\C^n$,
but we do not get (and do not need) the other 
approximation statements in that lemma.  


Combining the inequalities (\ref{est2}) 
and (\ref{small-increase2}) we obtain
\[
	v(x) \le \int_0^{2\pi} u(h(\E^{\I t})) \, \frac{\d t}{2\pi}
	<  \int^{2\pi}_0 v(f(\E^{\I t}))\, \frac{\d t}{2\pi} + 3\epsilon.
\]
Since this holds for every $\epsilon >0$, the 
property (\ref{eq:submean}) follows.
This proves that $v$ is plurisubharmonic on $X_{\reg}$
provided that Lemma \ref{RH2} holds.

\smallskip
\noindent
{\em Proof of Lemma \ref{RH2}.}
In the case when $X$ is a complex manifold, 
the first proof of Rosay  \cite{Rosay:Poletsky1}
uses a rather delicate construction of  Stein neighborhoods; 
this approach was developed further by 
L\'arusson and Sigurdsson \cite{Larusson-Sigurdsson2003}.
Later Rosay  \cite{Rosay:Poletsky2} gave another proof 
using an initial approximation by non-holomorphic discs 
with small $\dibar$-derivatives,  
approximating these by holomorphic discs, 
and finally patching the partial solutions
together by solving a nonlinear Cousin-I problem.
None of these methods seems to extend to complex 
spaces with singularities without major technical difficulties.
On the other hand, the method that we use here 
works in essentially the same way as in the nonsingular case.

Since the function $u$ is bounded on compacts,
it is a trivial matter to reduce the proof to the special situation
considered above so that the double integral in (\ref{est2})
changes by less than $\epsilon$; in the sequel we consider
this case. Let $\Delta_j$ be the discs chosen in the paragraph following
(\ref{est1}) (see Fig.\ \ref{Fig1}). Fix an index $j\in\{1,\ldots,l\}$.
We shall apply Lemma \ref{RH1} over $\ol\Delta_j$ 
to find an analytic disc $f'_j \colon \ol \Delta_j\to X$ 
that approximates $f$ uniformly as close as desired outside of a 
small neighborhood of the arc $\ol I_j$ 
in $\ol \Delta_j$ and  satisfies  the estimate
\begin{equation}
\label{local1}
	\int_{I_j} 
	u \bigl(f'_j (\E^{\I t})\bigr) \,\frac{\d t}{2\pi} < 
	\int^{2\pi}_0 \int_{I_j} 
       u\bigl( \xi(\E^{\I t},\E^{\I \theta})\bigr) \, 
       \frac{\d t}{2\pi} \frac{\d\theta}{2\pi}
       + \frac{|I_j|\,\epsilon}{2\pi}.
\end{equation}
Recall from (\ref{discs-xi}) that 
$\xi(\zeta,z)= g_j\bigl(\zeta,\chi(\zeta)z\bigr)$
for $\zeta \in \ol\Delta_j$ and $z\in \clD$.
Consider the function 
\[
	u'_j(\zeta,z)=u\bigl(g_j\bigl(\zeta,z)\bigr),\qquad
	\zeta\in \ol\Delta_j,\ z\in\clD
\]
and the smooth family of analytic discs in $\C^2_{(\zeta,z)}$ 
given by
\[
	g'_j(\zeta,z)=\bigl(\zeta,\chi(\zeta)\, z \bigr), \qquad 
	\zeta\in b\Delta_j,\ z\in\clD.
\]
Then 
\begin{equation}
\label{xi}
	\xi=g_j\circ g'_j\quad{\rm and}\quad u'_j=u\circ g_j
	\quad{\rm on}\quad b\Delta_j\times\clD.
\end{equation}
Applying Lemma \ref{RH1} with $\D$ replaced 
by $\Delta_j$, $u$ replaced by $u'_j$ and 
$g$ replaced by $g'_j$ furnishes an analytic disc
$h'_j\in \cA(\Delta_j,\C^2)$ which approximates 
the disc $\zeta\mapsto (\zeta,0)$ outside of a 
small neighborhood of the arc $\ol I_j$ and such that 
\begin{equation}
\label{local2}
	\int_{I_j} 
	u'_j \bigl(h'_j (\E^{\I t})\bigr) \,\frac{\d t}{2\pi} <
	\int^{2\pi}_0  \int_{I_j} 
       u'\bigl( g'_j(\E^{\I t},\E^{\I \theta})\bigr) \, 
       \frac{\d t}{2\pi} \frac{\d\theta}{2\pi}
        +  \frac{|I_j|\,\epsilon}{2\pi}.
\end{equation}
Since $g_j\colon \ol\Delta_j\times\clD\to X$ 
is holomorphic in the interior $\Delta_j\times\D$,  
the map
\begin{equation}
\label{fjprime}
	f'_j:=g_j\circ h'_j\colon \ol\Delta_j\to X
\end{equation}
is an analytic disc in $X$ such that 
$f'_j(\zeta)\approx g_j(\zeta,0)=f(\zeta)$
for $\zeta$ outside a small neighborhood of $\ol I_j$
in $\ol \Delta_j$. From (\ref{xi}) and (\ref{fjprime})
it follows that
\[
	u\circ f'_j= u\circ g_j \circ h'_j = u'_j\circ h'_j,
	\qquad
	u\circ\xi = u\circ g_j\circ g'_j = u'_j\circ g'_j
\]
hold on $b\Delta_j\times \clD$.
Hence the integrals in (\ref{local1}) equal
the corresponding integrals in (\ref{local2}), and so the disc $f'_j$
satisfies the desired properties.

If the approximation of $f$ by $f'_j$ is close enough 
for each $j=1,\ldots,l$, we can use Proposition \ref{gluing-sprays}
to glue this collection of discs into a single analytic disc 
$h\colon \ol{\,\D}\to X$ which approximates $f$ away from 
the union of arcs $\cup_{j=1}^l I_j$, and which approximates 
the disc $f'_j$ over a neighborhood of $\ol I_j$ for each $j$. 
In particular we can insure that for each $j=1,\ldots, l$ we have
\[
	\int_{I_j} u\bigl(h(\E^{\I t})\bigr) \,\frac{\d t}{2\pi} \approx
	\int^{2\pi}_0 \!\!\! \int_{I_j}
       u\bigl(g(\E^{\I t},\E^{\I\theta})\bigr) \, 
       \frac{\d t}{2\pi} \frac{\d\theta}{2\pi}.
\]	 
By adding up these terms and estimating the difference
over the remaining set $E=\bT\bs \cup_j I_j$ with small length
$|E|$ we obtain (\ref{small-increase2}).

Let us explain the gluing.
We apply Proposition \ref{sprays-exist} 
to embed $f$ as the central map $f_0$ in a dominating spray of discs 
$f_p\in \cA(\D,X)$ depending holomorphically
on a parameter $p$ in an open ball $P$ in some $\C^N$.
Next we apply the above approximation procedure 
simultaneously to all discs $f_p$ in the spray,
with a holomorphic dependence on the parameter.
(The obvious  details need not be repeated.)
This gives for every $j$ a holomorphic spray of discs 
$f'_{p,j} \in \cA(\Delta_j,X)$ ($p\in P)$  
approximating the spray $f_p$ over the complement 
of a small neigborhood of $\ol I_j$ in $\ol \Delta_j$. 
If the approximations are close enough, we can glue these 
sprays into a spray of discs $F_p \in \cA(\D,X)$
by using Proposition \ref{gluing-sprays}. 
(Here we allow the parameter set $P$ to shrink.)
For gluing we use a Cartan pair $(D_0,D_1)$
with $D_0\cup D_1=\D$ obtained as follows
(see Fig.\ \ref{Fig1}).
We get $D_0$ by denting the boundary circle $\bT=b\D$
slightly inward along each of the arcs $\ol I_j\subset \bT$.
The set $D_1$ equals $\cup_{j=1}^l \Delta_j$. 
Then the two sprays are close to each other over 
$\bar D_0\cap \bar D_1$, so Proposition  \ref{gluing-sprays} applies. 

The core disc $h=F_0$ obtained in this way is close  to 
the initial disc $f$ over the complement of a small neighborhood 
of $\cup_j \ol I_j$ in $\clD$, while over a neighborhood 
of $\ol I_j$ it is close to $f'_j$.
If the approximations are close enough then 
$h$ also satisfies the estimate (\ref{small-increase2})
in Lemma \ref{RH2}. 
\qed

\medskip
{\bf Step 5.}
We now prove that $v=\hat u$ is plurisubharmonic on all of $X$.
Let $w\colon X\to \R\cup\{-\infty\}$ be the upper regularization of $v|X_{\reg}$:
\[
		w(p)=\left \{\begin{array}{ll}
		v(p),& p\in X_{\reg};\cr
    \limsup_{q\in X_{\reg}, q\to p}v(q), & p\in X_{\sing}. \end{array} 
    \right.
\]
It is easy to see that $w\le u$. 
Since $X$ is normal, $w$ is plurisubharmonic on 
$X$ according to a result of Grauert and Remmert \cite[Satz 4]{Grauert-Remmert}. 

To complete the proof of the theorem we show that $v=w$ on $X_{\sing}$.
Fix a point $p\in X_{\sing}$ and a disc $f\in\cA(\D,X,p)$.
The fact that $w$ is plurisubharmonic implies that $w(p)\le P_w(f)$. 
Since $w\le u$, we also have $P_w(f)\le P_u(f)$. 
Therefore $w(p)\le P_u(f)$ for every $f\in\cA(\D,X,p)$ 
which implies that $w(p)\le v(p)$.

Suppose now that $w(p)< v(p)$; we shall reach a contradiction.
Choose $\epsilon>0$ so that $w(p)+3\epsilon< v(p)$.
Since $w$ is upper semicontinuous, there is a neighborhood $U$ of $p$
such that 
\begin{equation}
	w(q)\le w(p)	+\epsilon \quad {\rm for\ each}\ q\in U.
\label{nbd of p}
\end{equation}
We can choose an analytic disc $f\in\cA(\D,U,p)$ such that 
$f(\bT)\subset U\cap X_{\reg}$.  
Since $w=v$ on $X_{\reg}$, we have using (\ref{nbd of p}) 
\begin{equation}
P_v(f)=	P_w(f)\le \sup\{w\circ f(\zeta) \colon \zeta\in \bT\}\le w(p)+\epsilon.
\label{average f}
\end{equation}
For each point $\E^{\I t} \in \bT$ we choose  a disc 
$g_t=g(\E^{\I t},\cdotp) \in \cA(\D,X,f(\E^{\I t}))$ such that
$P_u(g_t) < v(f(\E^{\I t})) +\epsilon$.
As in Step 3 we can deform $g_t$ to a 
continuous family $g\colon \bT\times \clD\to X$
of analytic discs such that 
\begin{equation}
\label{eq:averaging1}
   \left| \int^{2\pi}_0 v\bigl(f(\E^{\I t})\bigr)\,\frac{\d t}{2\pi}
   -  \int^{2\pi}_0 \!\!\! \int^{2\pi}_0
       u\bigl(g(\E^{\I t},\E^{\I\theta})\bigr) \, 
       \frac{\d t}{2\pi} \frac{\d\theta}{2\pi} \right| <\epsilon.
\end{equation}
Lemma \ref{RH2} furnishes an analytic disc
$h\in \cA(\D,X,p)$ such that 
\[
	P_u(h) < 
			\int^{2\pi}_0 \!\!\! \int^{2\pi}_0
       u\bigl(g(\E^{\I t},\E^{\I\theta})\bigr) \, 
       \frac{\d t}{2\pi} \frac{\d\theta}{2\pi} + \epsilon.
\]
By (\ref{eq:averaging1}) and (\ref{average f}) we get
\[
	P_u(h)\le \int^{2\pi}_0 v\bigl(f(\E^{\I t})\bigr)\,\frac{\d t}{2\pi} 
	+ 2\epsilon\le   w(p)+3\epsilon< v(p)
\]
which contradicts the definition of $v$. 
This concludes the proof. 

\medskip
\textit{\bf Acknowledgement.}
The authors would like to thank the referee for his remarks.

Added to page proofs: Since the submission of this paper, the authors have used the  techniques developed in the present paper to obtain similar results for some other classes of disc functionals; see \cite{BDF2}.


\end{document}